\documentclass[11pt]{article}
\usepackage{t1enc}
\usepackage[latin1]{inputenc}
\usepackage[english]{babel}
\usepackage{amsmath,amsthm}
\usepackage{amsfonts}
\usepackage{latexsym}
\usepackage[dvips]{graphicx}
\usepackage{graphicx}
\usepackage{float}
\usepackage[natural]{xcolor}
\usepackage{algorithm}
\usepackage{algorithmic}
\usepackage{enumerate}
\usepackage{appendix}
\usepackage{multirow}
\usepackage{xcolor}
\usepackage[colorlinks,linkcolor=blue]{hyperref}
\usepackage{setspace}
\usepackage{comment}
\usepackage{fullpage}
\usepackage{amssymb} 
\usepackage{listings}
\newtheorem{theorem}{Theorem}\numberwithin{theorem}{section}
\newtheorem{lemma}[theorem]{Lemma}\numberwithin{theorem}{section}
\newtheorem{proposition}[theorem]{Proposition}
\numberwithin{theorem}{section}

\newtheorem{problem}[theorem]{Problem}

\newtheorem{definition}[theorem]{Definition}
\newtheorem{corollary}[theorem]{Corollary}

\def\eps{\varepsilon}

\def\int{\textrm{int}}
\numberwithin{equation}{section}

\begin{document}

\onehalfspacing
\title{
The perturbation threshold of degenerate graphs}

\author{
Jie Han$^a$, Seonghyuk Im${}^{b}{}^{c}$, Bin Wang$^a$, and 
Junxue Zhang$^a$ \\
\footnotesize $^a$ School of Mathematics and Statistics, Beijing Institute of Technology,  China\\ 
\footnotesize $^b$ Department of Mathematical Science, KAIST, South Korea\\
\footnotesize $^c$ Extremal Combinatorics and Probability Group (ECOPRO), Institute for Basic Science (IBS), South Korea \\
\footnotesize  han.jie@bit.edu.cn, seonghyuk@kaist.ac.kr, bin.wang@bit.edu.cn, jxuezhang@163.com}

\date{}


\maketitle 
\begin{abstract}

We show that for any $d\ge 2$ and $\Delta>0$ there exists $\eta>0$ such that the following holds:
Let $G$ be an $n$-vertex graph with at least $\Omega(n^2)$ edges and let $H$ be an $n$-vertex $d$-degenerate graph with maximum degree at most $\Delta$. 
Then with high probability, $G \cup G(n, n^{-1/d - \eta})$ contains a copy of $H$.
We also prove that the same conclusion extends to $d$-regular graphs with  $d\ge 4$ satisfying a certain edge expansion property, with the threshold improved to  $n^{-2/d - \eta}$. Such a property is satisfied by almost all
 $d$-regular graphs and for even $d$, by the $(d/2)$-th power of a Hamilton cycle.

\end{abstract}

\section{Introduction}

A binomial random graph or Erd\H{o}s--R\'{e}nyi graph $G(n,p)$ is a graph on the vertex set $[n]:=\{1,2,\ldots,n\}$ where each pair of vertices forms an edge independently with probability $p$.
A central topic in random graph theory is the study of thresholds for graph properties, defined as critical probabilities at which a random graph typically acquires a given property. 
Formally, given a graph property $\mathcal{P}$, we say that $G(n,p)$ has the property $\mathcal{P}$ with high probability (w.h.p.) if $\lim_{n\rightarrow \infty} \mathbb{P}[G(n,p)\in \mathcal{P}]=1$. 
A function $\hat{p}: \mathbb{N}\rightarrow [0,1]$ is called a \emph{threshold} for the property $\mathcal{P}$ if w.h.p. $G(n,p)$ has the property $\mathcal{P}$ when $p=\omega(\hat{p})$ and w.h.p. $G(n,p)$ does not have the property $\mathcal{P}$ when $p=o(\hat{p})$.

A classical result of Bollob\'{a}s and Thomason \cite{bollobas1987threshold} establishes that every monotone property admits a threshold. In particular, the property of containing a specific spanning subgraph $H$ is monotone and therefore has a threshold. 
A notable example is a result of Erd\H{o}s and R\'{e}nyi \cite{erdHos1966existence}, who proved that the threshold for the existence of a perfect matching in $G(n,p)$ is $\frac{\log n}{n}$. This spurred extensive research into determining thresholds for the containment of various fixed spanning structures in $G(n,p)$.

\subsection{Randomly perturbed graphs}
Bohman, Frieze, and Martin \cite{MR1943857} introduced the model of \emph{randomly perturbed graphs} $G\cup G(n,p)$, defined as the union of a deterministic graph $G$ and a binomial random graph $G(n,p)$ on the same vertex set $V(G)$. 
Let $\alpha >0$ be a constant and let $G_\alpha$ be an $n$-vertex graph with minimum degree at least $\alpha n$. 
Bohman, Frieze, and Martin \cite{MR1943857} proved that for every $\alpha >0$, there exists $C=C(\alpha)>0$ such that if $p=C/n$, then for any graph $G_\alpha$, w.h.p.~$G_\alpha \cup G(n,p)$ is Hamiltonian. 
Note that the threshold for Hamiltonicity in $G(n,p)$ alone is $\frac{\log n}{n}$; thus, the addition of the edges in $G_\alpha$ significantly reduces the required edge probability.
This reduced probability is called the \emph{perturbation threshold} for the graph property.
This result initiated a broad line of research on various spanning structures in randomly perturbed graphs, such as bounded degree spanning trees, powers of Hamilton cycles and so on~\cite{HPH,MR3922775,Bennett2017AddingRE,MR4025392,MR4130332,MR4052848,HPER,MR4052851,MR3595872}.

In 2020, B\"{o}ttcher, Montgomery, Parczyk, and Person \cite{MR4130332} proved a result on embedding general bounded degree graphs in randomly perturbed graphs.
They proved that for every $\Delta \geq 5$, if $H$ is an $n$-vertex graph with maximum degree at most $\Delta$, then for any $\alpha >0$, there exists a constant $C=C(\alpha, \Delta)$ such that w.h.p. $G_\alpha\cup G(n,Cn^{-\frac{2}{\Delta +1}})$ contains a copy of $H$. 
Again, the edge probability $n^{-\frac{2}{\Delta +1}}$ is significantly smaller than the threshold in $G(n,p)$ alone, which is $n^{-\frac{2}{\Delta+1} }(\log n)^{-\frac{2}{\Delta(\Delta+1)} }$, as proved by Frankston, Kahn, Narayanan, and Park~\cite{MR4298747}.

Another commonly studied family of sparse graphs is the family of $d$-degenerate graphs.
Note that $K_{1, n-1}$ is $1$-degenerate but it is unlikely to exist in $G_\alpha \cup G(n, p)$ unless $p$ is close to $1$.
Thus, it is natural to consider $d$-degenerate graphs with additional maximum degree conditions.
For the case when $d=1$, i.e., bounded degree forests, Krivelevich, Kwan, and Sudakov \cite{MR3595872} proved that for any $\alpha >0$ and $\Delta \geq 2$, there exists $C=C(\alpha, \Delta)$ such that for every $n$-vertex tree $T$ with $\Delta(T) \leq \Delta$, w.h.p. $G_\alpha\cup G(n, C/n)$ contains a copy of $T$.
This was extended by B\"{o}ttcher, Han, Kohayakawa, Montgomery, Parczyk, and Person~\cite{MR4025392}, who proved that w.h.p. $G_\alpha\cup G(n, C/n)$ contains all bounded degree $n$-vertex trees simultaneously.
Our first main theorem shows that for any $d \geq 2$, the threshold for bounded degree $d$-degenerate graphs in randomly perturbed graphs is polynomially smaller than $n^{-1/d}$, which is the threshold in $G(n, p)$ alone (shown by Riordan~\cite{MR1762785} for $d\ge 3$ and by Chen, Han, and Luo~\cite{chen2024thresholds} for $d=2$).

\begin{theorem}
  \label{maindegenerate}
    For every integer $d\geq2$ and constants $\varepsilon, \Delta > 0$, there exists $\eta>0$ such that the following holds. 
    If $G$ is an $n$-vertex graph with at least $\varepsilon n^2$ edges and $H$ is an $n$-vertex $d$-degenerate graph with maximum degree at most $\Delta$, then w.h.p. $G \cup G(n, n^{-\frac{1}{d}-\eta})$ contains a copy of $H$. 
\end{theorem}

We note that in this theorem, we only require that the base graph $G$ has at least $\Omega(n^2)$ edges instead of a linear minimum degree condition.
This aligns with the recent results by the authors~\cite{han2025perturbation}, which replaces the minimum degree condition on $G$ with a much weaker density condition for embedding $K_r$-factors, powers of Hamilton cycles, and bounded degree graphs in randomly perturbed graphs.

Our second main theorem considers a family of $d$-regular graphs with certain expansion property.
Note that a $K_{d+1}$-factor is a $d$-regular graph and its threshold in randomly perturbed graphs is $n^{-2/(d+1)}$, as shown in~\cite{MR3922775}.
We improve this bound when $H$ is far from being a disjoint union of $K_{d+1}$'s in the sense that it has a ``good'' edge-connectivity property.
For a set $X\subseteq V(H)$, we denote by $\partial(X)$ the set of edges with exactly one endpoint in $X$.
\begin{theorem}
  \label{mainthm:regular}
    For every integer $d\geq 3$ and constants $0<\varepsilon, \gamma\le 1/2$, there exists a constant $\eta>0$ such that the following holds for sufficiently large $n$. 
    Let $H$ be an $n$-vertex $d$-regular graph satisfying $|\partial(X)|\ge d+1$ for every $X\subseteq V(H)$ with $2\le |X|\le \gamma n$.
    Then for any $n$-vertex graph $G$ with at least $\varepsilon n^2$ edges, w.h.p. $G\cup G(n,n^{-\frac{2}{d}-\eta})$ contains a copy of $H$. 
\end{theorem}

It is shown in \cite{chen2024thresholds} that the threshold of such graphs $H$ is $n^{-2/d}$, and in a recent paper, Zhukovskii \cite{zhukovskii2025sharp} determined the sharp threshold for most of such $H$, which is $(1+o(1))(e/n)^{2/d}$.
Therefore, similar to Theorem \ref{maindegenerate}, Theorem \ref{mainthm:regular} shows a saving of a polynomial factor on the perturbation threshold for this family of graphs.

We note that the assumption $|\partial(X)|\ge d+1$ in Theorem~\ref{mainthm:regular} is sharp. Indeed, for $d\ge 3$ and $n\in d \mathbb{N}$, consider the following construction given in \cite{chen2024thresholds}.
Let $H$ be a $d$-regular graph 
obtained by taking $n/d$ disjoint copies of $K_{d}$. 
Partition each clique into sets $A_i,B_i$ of sizes $\lceil d/2\rceil$ and $\lfloor d/2\rfloor$. 
Add a perfect matching  between $A_i$ and $A_{i+1}$ for odd $i\in [n/d-1]$,
and between $B_i$ and $B_{i+1}$ for even $i\in [n/d-1]$. If $n/d$ is even, then add a  perfect matching  between $B_{n/d}$ and $B_{1}$; if $n/d$ is odd, then add one between $A_{n/d}$ and $B_{1}$. (When $d$ is odd,  $n/d$ is even  by the handshaking lemma.)
It is readily checked that $|\partial(X)|\ge d$ for every $X\subseteq V(H)$ with $2\le |X|\le n/2$. However, the perturbation threshold of $H$ is at least that of a $K_{d}$-factor, which by Balogh, Treglown and Wagner \cite{MR3922775} is $n^{-2/d}$, and in constrast, the threshold of $H$ is at most $n^{-2/d}\log n$. Consequently, the saving for such $H$ is at most logarithmic, not polynomial.

We also present two applications of Theorem \ref{mainthm:regular}. 
The first result states that for almost all $d$-regular graphs with $d\ge 4$, the condition in Theorem~\ref{mainthm:regular} is satisfied; thus, a typical $d$-regular graph is much easier to embed in randomly perturbed graphs than the $K_{d+1}$-factor.
\begin{corollary}\label{cor:almost_2dreg}
  Let $d \geq 4$ be a fixed integer. Then there exists a constant $\eta>0$ such that the following holds for sufficiently large $n$.
  For almost all $n$-vertex $d$-regular graphs $H$ and any $n$-vertex graph $G$ with at least $\varepsilon n^2$ edges, w.h.p. $G\cup G(n,n^{-\frac{2}{d}-\eta})$ contains a copy of $H$.
\end{corollary}

This follows from \cite[Theorem 7.32]{Bollobas2001random}, which implies that w.h.p. 
a random $d$-regular graph has an edge-cut of size at most $d$ only when one side of the cut is a singleton. We prove this formally in Section \ref{sec:cor}.
We also note that when $d=3$, with probability $\Omega(1)$, a random $3$-regular graph contains a triangle, and thus there exists a set $X$ of three vertices with $|\partial(X)|=3$.


Another notable family of graphs satisfying the condition in Theorem~\ref{mainthm:regular} is the $d$-th power of a Hamilton cycle. The \emph{$d$-th power of a Hamilton cycle} is the graph obtained from a Hamilton cycle $C_n$	
  by adding an edge between every pair of vertices whose distance along $C_n$ 	
  is at most $d$.
B\"ottcher, Montgomery, Parczyk, and Person~\cite{MR4130332} proved that for every $d \geq 2$ and $\alpha >0$, there exists a constant $\eta=\eta(\alpha, d)$ such that w.h.p. $G_\alpha \cup G(n, n^{-1/d-\eta})$ contains the $d$-th power of a Hamilton cycle.
The minimum degree condition on $G_\alpha$ was recently relaxed to a density condition by the authors~\cite{han2025perturbation}.
Theorem~\ref{mainthm:regular} recovers both results.
\begin{corollary}\label{cor:Hcyclepower}
  For every integer $d\geq2$ and constant $\varepsilon>0$, there exists a constant $\eta>0$ such that the following holds. 
  For any $n$-vertex graph $G$ with at least $\varepsilon n^2$ edges, w.h.p. $G\cup G(n,n^{-\frac{1}{d}-\eta})$ contains the $d$-th power of a Hamilton cycle.
\end{corollary}

One common feature of Theorems~\ref{maindegenerate} and \ref{mainthm:regular} is that both the perturbation thresholds enjoy a saving of a polynomial factor in $n$, in contrast to most other results on perturbation thresholds where the saving is a poly-logarithmic factor.
We now present the main technical theorem that we use to derive both Theorems~\ref{maindegenerate} and \ref{mainthm:regular}.

Before stating our technical theorem, we introduce some notation.
For a digraph $D$ and a vertex $v \in V(D)$, let $N^+_{D}(v)=\{u\in V(D)\setminus\{v\}:(v,u) \text{ is an arc in } D\}$ be the \emph{out-neighborhood} of $v$ in $D$.
This concept extends to any vertex set $W\subseteq V(D)$ by defining $N^+(W)=\bigcup_{w\in W} N^+(w)\setminus W$. 
For a vertex $v\in V(D)$, the \emph{$p$-out-ball} $B_{D}^{+p}(v)$ of $v$ is the set of vertices reachable from $v$ by a directed path of length at most $p$, including $v$ itself.
We omit the subscript $D$ when it is clear from the context.
Given a graph $F$, let $v_F$ and $e_F$ denote the number of vertices and edges of $F$, respectively.
If $F$ has at least two vertices, its 1-\emph{density} is defined as $m_1(F):=\max\left\{ d(F'):F'\subseteq F, v_{F'}\geq 2 \right\}, \text{ where } d(F'):=\frac{e_{F'}}{v_{F'}-1}.$ 

We use $\ll$ to denote a hierarchy between constants.
If we write that a statement holds whenever $0 < a \ll b, c \ll d$, then it means that there exist non-decreasing functions $g_1, g_2 \colon (0,1] \to (0,1]$ and $f \colon (0,1]^2 \to (0,1]$ such that the statement holds for all $a, b, c, d$ satisfying $b \le g_1(d)$, $c \le g_2(d)$, and $a \le f(b, c)$.
We will not explicitly compute these functions to avoid cluttering the presentation of the proofs.
With these definitions, we can now state our main technical result.

\begin{theorem}
\label{combin}
Let $d>1$ be a real number. Suppose $1/n \ll \eta \ll \varepsilon' \ll 1/K\ll\varepsilon,1/d,1/\Delta$.
 Let $H$ be an $n$-vertex graph with $\Delta(H)\leq\Delta$ and let $D$ be an acyclic orientation of $H$.
 Let $V':=\{v\in V(H):|B_{D}^{+K}(v)|\geq K/2\}$. Suppose that every subgraph $H'$ of $H$ of order $m$ satisfies the following:
 \begin{itemize}
   \item[(1)] If $m>K/2$, then $d(H')\leq d-\varepsilon'$ or there exists $v\in V(H')\cap V'$ with $E(D[B^{+(K+1)}(v)])\subseteq E(H')$.
   \item[(2)] If $m\leq K/2$, then $e_{H'}\leq d(m-1)-1/2$.
 \end{itemize}
 If $G$ is an $n$-vertex graph with at least $\varepsilon n^2$ edges and $p=n^{-\frac{1}{d}-\eta}$, then w.h.p. $G\cup G(n,p)$ contains a copy of $H$.
\end{theorem}

For an $n$-vertex $d$-degenerate graph $H$, we have $m_1(H)\le d-o_n(1)$, and its threshold (in $G(n,p)$) is equal to $n^{-{1}/{d}} \asymp n^{-{1}/{m_1(H)}}$.
Indeed, the celebrated result of Frankston, Kahn, Narayanan and Park~\cite{MR4298747} indeed gives that the threshold of any graph $H$ is at most $n^{-{1}/{m_1(H)}}\log n$, which makes $n^{-{1}/{m_1(H)}}$ a natural target for the $H$-containment property for many graphs $H$.

For the randomly perturbed model, given that $\Delta(H)\le \Delta$, it is clear that the deterministic graph $G$ with density $\eps$ (or minimum degree $\eps n$) may contribute at most $o(n)$ edges to a copy of $H$ in $G\cup G(n,p)$.
Therefore, one must embed a subgraph $H^*$ of $H$ to $G(n,p)$ with $e_{H^*}= e(H) - o(n)$.
It implies that the perturbation threshold of $H$ is upper bounded by the threshold of $H^*$, which in turn is upper bounded by $n^{-1/m_1(H^*)}$.
Thus, roughly speaking, to have a polynomial-saving on the perturbation threshold of $H$, we need to have
\[
m_1(H^*) \le m_1(H) - \eps
\]
for some absolute $\eps >0$.
This suggests a proof strategy: in the proof, we would like to choose $H^*$ (equivalently, to choose the edges to be covered in $G$) so that every induced subgraph of $H^*$ has density at most $m_1(H) - \eps$.
For large subgraphs, this is achieved by removing some edges of $H$ from it; for small subgraphs (of constant order), as we only remove $o(n)$ edges from $H$, they must have density at most $m_1(H) - \eps$ in $H$ by themselves. 

The assumptions (1) and (2) in Theorem~\ref{combin} provide a sufficient condition: if i) every large subgraph of $H$ must either have a lower density or completely contain an (out)-$K$-ball, ii) every small subgraph has a strictly lower density, then the polynomial-saving on the perturbation threshold is guaranteed.
In view of the above discussion, (2) is essentially necessary while it is not clear to us whether (1) is necessary for the polynomial-saving. 

The rest of this paper is structured as follows.
Section \ref{sec:pre} introduces our notation for digraphs and preliminary results on spreadness. Section \ref{sec:proffram} presents the proof of our main framework, Theorem \ref{combin}. The applications of this framework are then developed in Section \ref{sec:profdegene}, where we prove Theorem \ref{maindegenerate} for degenerate graphs and Theorem \ref{mainthm:regular} for regular graphs. Corollaries~\ref{cor:almost_2dreg} and \ref{cor:Hcyclepower}, which follow from Theorem \ref{mainthm:regular}, are stated in Section \ref{sec:cor}.

\section{Notation and Preliminaries}\label{sec:pre}
\subsection{Basic notation}
Given a graph $F$, let $v_F$ and $e_F$ denote the number of vertices and edges of $F$, respectively.
The \emph{edge cut} of $F$ associated with $X \subseteq V(F)$, denoted by $\partial(X)$, is the set of edges with one end in $X$ and the other in $V(F) \setminus X$. 
For two subsets $X,Y\subseteq V(F)$, we define $E(X)$  as  the set of edges with both ends in $X$ and $E(X,Y)$ as the set of edges with one end in $X$ and the other  in $Y$.

A digraph $D$ consists of a non-empty finite set $V(D)$ of elements called \emph{vertices} and a finite set $A(D)$ of ordered pairs of distinct vertices called \emph{arcs}.
The \emph{order}(\emph{size}) of $D$ is the number of vertices (arcs) in $D$.
The order of $D$ will sometimes be denoted by $|D|$. Denote the edge set of the underlying graph of $D$ by $E(D)$.
For a subset $S\subseteq A(D)$, let $D\setminus S$ denote the digraph obtained from $D$  by deleting all arcs in $S$.
Denote by $D[S]$  the \emph{arc-induced subdigraph} with arc set $S$ and vertex set consisting of all vertices incident with arcs in $S$. 
For a vertex $v$ in $D$, $N^+_{D}(v)=\{u\in V(D)\setminus\{v\}:(v,u)\in A(D)\}$,
$N^-_{D}(v)=\{w\in V(D)\setminus\{v\}:(w,v)\in A(D)\}$.
The sets $N^+_{D}(v)$, $N^-_{D}(v)$ are called the \emph{out-neighborhood}, \emph{in-neighborhood} of $v$, respectively.
We call the vertices in $N^+_{D}(v)$, $N^-_{D}(v)$ the \emph{out-neighbors}, \emph{in-neighbors} of $v$.
For any subset $X$ of $V(D)$, let $N_D^+(v,X)$ ($N_D^-(v,X)$) denote the out-neighbors (in-neighbors) of $v$ in $X$.
For a set $W\subseteq V(D)$, let $N^+_{D}(W)=\bigcup_{w\in W}N^+_{D}(w)\setminus W$ 
and $N^-_{D}(W)=\bigcup_{w\in W}N^-_{D}(w)\setminus W$.
A \emph{walk} in $D$ is an alternating sequence $W=x_1a_1x_2a_2x_3\cdots x_{k-1}a_{k-1}x_k$ of vertices $x_i$ and arcs $a_j$ from $D$ such that $a_i=(x_i,x_{i+1})$.
If the vertices and arcs are distinct, then $W$ is a \emph{path}.
Furthermore, if $x_1=x_k$, then $W$ is a \emph{cycle}.
The \emph{length} of a walk is the number of its arcs.
The \emph{distance} between $x$ and $y$, denoted by dist$(x,y)$, is the minimum length of an $(x,y)$-walk.
A digraph is called \emph{acyclic} if it has no cycle.


An \emph{oriented} graph is a digraph with no cycles of length two.
A digraph $D$ is called an \emph{orientation} of graph $H$ if it is obtained from $H$ by replacing each edge $\{x,y\}$ of $H$ by $(x,y)$ or $(y,x)$.
A set $Q$ of vertices in a digraph $D$ is \emph{independent} if $A(D[Q])=\emptyset$.



\subsection{Spreadness}

The following notion of spreadness has played a critical role in the recent celebrated results on the fractional expectation thresholds by Frankston, Kahn, Narayanan and Park \cite{MR4298747}.
%
A \emph{hypergraph} $H$ consists of a vertex set $V(H)$ and edge set $E(H)\subseteq 2^{V(H)}$.

\begin{definition}[Spread]
Let $q\in[0,1]$ and $r\in \mathbb{N}$. Assume that $H$ is a hypergraph on the vertex set $V$ and $\mu$ is a probability measure on the edge set of $H$. We say that $\mu$ is $q$-$spread$ if for every $S\subseteq V$, the following holds:
\[
\mu(\{A\in E(H):S\subseteq A\})\leq q^{|S|}.
\]
\end{definition}

 Frankston, Kahn, Narayanan and Park \cite{MR4298747} established a connection between spreadness and the threshold in random graphs.
 We use $V_p$ to denote a subset of $V$ where each $x\in V$ is included independently with probability $p$.
 \begin{proposition}
 \label{FKNP}{\rm(\cite{MR4298747}, Theorem 1.6)}
Let $\mathcal{H}$ be an $r$-bounded hypergraph on vertex set $V$ that supports a $q$-spread distribution.
If $p \geq Kq \log r$, then with probability $1-o_r(1)$, the random subset $V_p$ contains an edge of $\mathcal{H}$.
 \end{proposition}

Pham, Sah, Sawhney and Simkin \cite{pham2023toolkit} introduced a notion of vertex-spreadness, and 
Kelly, M\"{u}yesser and Pokrovskiy
\cite{KELLY2024507} put it in  a general setting.
\begin{definition}[Vertex-spread]
Let $X$ and $Y$ be finite sets and let $\mu$ be a probability distribution over injections $\psi:X\rightarrow Y$.
For $q\in[0,1]$, we say that $\mu$ is a $q$-$vertex$-$spread$ if for every $s\leq |X|$ and every two sequences of distinct vertices $x_1,\ldots,x_s\in X$ and $y_1,\ldots,y_s\in Y$,
\[
\mu(\{\varphi:\varphi(x_i)=y_i{\rm \ for\ all\ } i\in[s]\})\leq q^s.
\]
\end{definition}
A \emph{hypergraph embedding} $\psi:G\rightarrow H$ of a hypergraph $G$ into a hypergraph $H$ is an injective map $\psi:V(G)\rightarrow V(H)$ that maps edges of $G$ to edges of $H$, so there is an embedding of $G$ into $H$ if and only if $H$ contains a subgraph isomorphic to $G$.
Note that when  $H$ is a complete hypergraph, and $G$ and $H$ have the same vertex set, the uniformly random embedding $\psi:G\rightarrow H$ is a permutation of $V(H)$, which is $e/v_H$-vertex-spread (by Stirling approximation).
The following result allows us to connect spreadness and vertex-spreadness.
\begin{proposition}\label{kelly}
{\rm (\cite{KELLY2024507}, Proposition 1.17)}
For every $k,\Delta\in  \mathbb{N}$ and $C>0$, there exists a constant $C_{\ref{kelly}}>0$ such that the following holds for sufficiently large $n$.
Let $H$ and $G$ be $n$-vertex $k$-graphs.
If there is a $(C/n)$-vertex-spread distribution on embeddings $G\rightarrow H$ and $\Delta(G)\leq\Delta$, then there is a $(C_{\ref{kelly}}/n^{1/m_1(G)})$-spread distribution on subgraphs of $H$ which are isomorphic to $G$.
\end{proposition}

Furthermore, combined with Proposition \ref{FKNP}, it gives an upper bound $n^{-1/m_1(G)}\log n$ for the threshold of $G$.


\section{Proof of Theorem~\ref{combin}}
\label{sec:proffram}
Our proof of Theorem~\ref{combin} is divided into the following two steps.

\noindent
\textbf{Step 1.} Let $D$ be an orientation of $H$ satisfing the assumption of Theorem~\ref{combin}. Then there is an independent set $X$ of size $o(n)$ such that the $1$-denstity of the subgraph of $H$ obtained by deleting a set of vertex-disjoint in-stars with centers in $X$ is at most $d-\Omega(1)$ (see Lemma~\ref{m1H}).

\noindent
\textbf{Step 2.} If there exists such an independent set $X$ of $H$, then $G\cup G(n,n^{-1/d-\Omega(1)})$ contains a copy of $H$ (see Lemma \ref{pt}) where $G$ is a deterministic graph with positive density.

Before our proofs, we introduce some additional notation.

A digraph $S^+ = (V, A)$ is called an \emph{out-star} if there exists a unique vertex $v \in V$ such that $A = \{(v, u) : u \in V \setminus \{v\}\}$. 
The vertex $v$ is called the \emph{center} of $S^+$.
The vertices $V\setminus\{v\}$ are called the \emph{leaves} of $S^+$.
Given a digraph $D$ and  an independent set $Q$ in $D$, let $S(Q)$ be the subgraph with the arc set $A(V(D)\setminus Q, Q)$.
Note that $S(Q)$ is a collection of out-stars with centers in $V(D)\setminus Q$ and leaves in $Q$.
For each such out-star in $S(Q)$, we choose exactly one arc towards a leaf and let $M^-(Q)$ be the set of all chosen arcs (See Figure \ref{outstar}). 
Similarly, a digraph $S^- = (V, A)$ is called an \emph{in-star} if there exists a unique vertex $v \in V$ such that $A = \{(u, v) : u \in V \setminus \{v\}\}$. The vertex $v$ is called the \emph{center} of $S^-$.

\begin{figure}[htp]
  \centering
  \includegraphics[width=8cm]{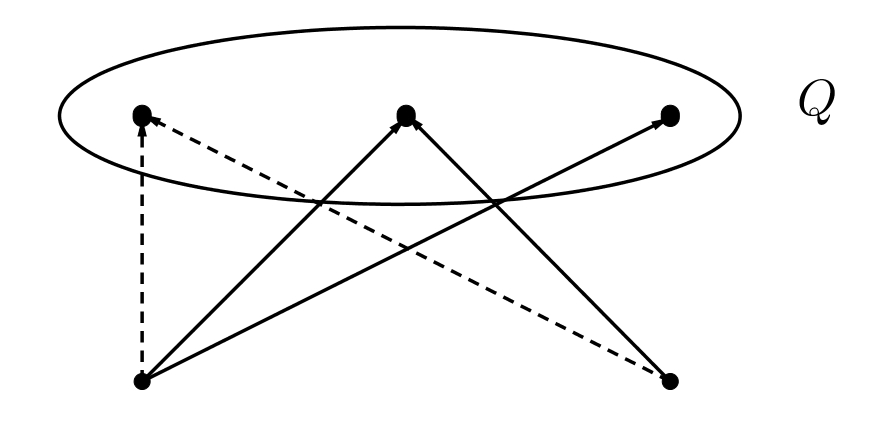}\\
  \caption{$S(Q)$ where dotted arcs are edges of $M^-(Q)$}
  \label{outstar}
\end{figure}

\begin{lemma}
\label{m1H}
 Let $d>1$ be a real number.
 Suppose $1/n\ll \varepsilon'\ll 1/K\ll\gamma \ll 1/d,1/\Delta$.
 Let $H$ be a graph with $\Delta(H)\leq\Delta$ and $D$ be an acyclic orientation of $H$.
 Let $V':=\{v\in V(H),|B_{D}^{+K}(v)|\geq K/2\}$. Suppose that every subgraph $H'$ of $H$ of order $m$ satisfies:
 \begin{itemize}
   \item[(1)] If $m>K/2$, then $d(H')\leq d-\varepsilon'$ or there exists $v\in V(H')\cap V'$ with $E(D[B^{+(K+1)}(v)])\subseteq E(H')$.
   \item[(2)] If $m\leq K/2$, then $e_{H'}\leq d(m-1)-1/2$.
 \end{itemize}
Then there exists an independent set $X\subseteq V(H)$ of size at most $\gamma n/2$ such that $m_1(H\setminus M^-(X))\leq d-\varepsilon'$.
\end{lemma}
\begin{proof}
Set $K:=10\ln (20/\gamma)/\gamma$.
Let $V=V(H)$ and $V':=\{v\in V(H),|B_{D}^{+K}(v)|\geq K/2\}$.
Let $V_{\gamma/4}$ be a subset of $V$ where each vertex is included independently with probability $\gamma/4$.
Clearly $ \mathbb{E}[|V_{\gamma/4}|]=\gamma n/4$.
By Chernoff's bound (cf. \cite[Corollary 2.3]{Random}), we have
\begin{equation}
\label{size}
\mathbb{P}[|V_{\gamma/4}|\leq3\gamma n/8]=1-o(1).
\end{equation}
    For any $v\in V'$, it is easy to see that
    \[
    \mathbb{P}\left[(B_{D}^{+K}(v) \setminus \{v\}) \cap V_{\gamma/4} =\emptyset\right]=(1-\gamma/4)^{|B_{D}^{+K}(v)|-1}\leq e^{-\frac{\gamma}{4}(|B_{D}^{+K}(v)|-1)}\leq e^{-\frac{\gamma \cdot K}{10}},
    \]
   Let $Z$ be the number of vertices $v$ in $V'$ such that $(B_{D}^{+K}(v) \setminus \{v\}) \cap V_{\gamma/4}=\emptyset$.
   Then
    \[
    \mathbb{E}[Z]\leq e^{-\frac{\gamma \cdot K}{10}}|V'|.
    \]
    By Markov's inequality, we have
    \[
    \mathbb{P}[Z\geq2e^{-\frac{\gamma \cdot K}{10}}|V'|]\leq1/2.
    \]
   Combining with \eqref{size}, there exists a set $V_{\gamma/4}$ satisfing the following properties:
   \begin{itemize}
     \item[$(Q1)$] $|V_{\gamma/4}|\leq3\gamma n/8$,
     \item[$(Q2)$] For all but at most $2e^{-\frac{\gamma \cdot K}{10}}|V'|$ vertices $v$ of $V'$, we have $(B_{D}^{+K}(v) \setminus \{v\}) \cap V_{\gamma/4}\neq\emptyset$.
   \end{itemize}

 For each vertex $v\in V'$ such that $(B_{D}^{+K}(v) \setminus \{v\}) \cap V_{\gamma/4}=\emptyset$, we choose a vertex $u \in (B_{D}^{+K}(v) \setminus \{v\})$ arbitrarily and add it to a set $X_0$. 
 Note that $|X_0|\leq 2e^{-\frac{\gamma \cdot K}{10}}|V'|\leq 2e^{-\frac{\gamma \cdot K}{10}}n\leq\gamma n/8$ since $K=10\ln (20/\gamma)/\gamma$.
 Let $X_1:=V_{\gamma/4}\cup X_0$. By construction, $X_1$ satisfies the following properties:
    \begin{itemize}
     \item[$(Q1')$] $|X_1|\leq\gamma n/2$,
     \item[$(Q2')$] For each vertex $v$ of $V'$, we have $(B_{D}^{+K}(v) \setminus \{v\}) \cap X_1\neq\emptyset$.
    \end{itemize}
 Next, we use the following process.
 
 \noindent
 \textbf{To output an independent set}: Initiate $X_1=V_{\gamma/4}\cup X_0$, $B_1=\{v\in X_1:|N_{D[X_1]}^+(v)|=0\}$ and $i=1$.
 Note that $D[X_1]$ is also acyclic.

 \emph{Step 1}. Define $X_{i+1}=X_i\setminus N^-(B_i)$.

 \emph{Step 2}. Define $B_{i+1}=\{v\in X_{i+1}:|N_{D[X_{i+1}]}^+(v)|=0\}$.

 \emph{Step 3}. If $X_{i+1}$ is not independent in $D$, then update $i:=i+1$ and go to Step 1; otherwise, terminate the process.

 We claim that after the process, we obtain an independent set $X$ satisfying the following.
  \begin{itemize}
     \item[$(P1)$] $|X|\leq\gamma n/2$.
     \item[$(P2)$] For all vertices $v$ of $V'$, we have $(B_{D}^{+(K+1)}(v) \setminus \{v\}) \cap X\neq\emptyset$.
   \end{itemize}

To see this, we first show that the process terminates.
Suppose that $X_i$ is not an independent set in $D$.
Then as $D[X_i]$ is acyclic, there exists an edge $(u,v)\in A(D[X_i])$ such that $v\in B_i$.
Thus, $u \notin X_{i+1}$ so $|X_{i+1}| < |X_i|$.
Hence, the process terminates in at most $|X_1|$ steps.

We now prove $(P1)$ and $(P2)$. 
First, we obtain $(P1)$ by $(Q1')$ and $|X|\leq|X_1|$.
For $(P2)$, by $(Q2')$, $(B_{D}^{+K}(v) \setminus \{v\}) \cap X_1\neq\emptyset$ for any vertex $v\in V'$.
If $(B_{D}^{+K}(v) \setminus \{v\})\cap X\neq\emptyset$, then it is obvious that $B_{D}^{+(K+1)}(v)\cap X\neq\emptyset$.
If $(B_{D}^{+K}(v) \setminus \{v\})\cap X=\emptyset$, then the vertices in $(B_{D}^{+K}(v) \setminus \{v\})\cap X_1$ were deleted in the above process.
We choose an arbitrary vertex $u$ in $(B_{D}^{+K}(v) \setminus \{v\})\cap X_1$. Then there exists an index $i$ and a vertex $w \in B_i$ such that $(u,w)$ is an arc in $D$.
Since $w \in B_j$ for all $j \geq i$, $w$ is preserved in $X$.
Thus, $w \in  B_{D}^{+(K+1)}(v) \cap X$.
In addition, $w$ is different from $v$ since $D$ is acyclic.
Therefore, $(B_{D}^{+(K+1)}(v) \setminus \{v\}) \cap X\neq\emptyset$ for all $v \in V'$, giving $(P2)$.

We now show that $X$ satisfies the conclusion of the lemma.
Let $H'$ be a subgraph of $H\setminus M^-(X)$ of order $m$.

(1) If $m>K/2$, then we claim that there is no vertex $v\in V(H') \cap V'$ such that $E(D[B_{D}^{+(K+1)}(v)])\subseteq E(H')$.
For all vertices $v$ of $V'$, we have $(B_{D}^{+(K+1)}(v) \setminus \{v\}) \cap X\neq\emptyset$ by (P2).
We choose $u\in B_{D}^{+(K+1)}(v) \setminus \{v\}$ and $u^{-} \in B_{D}^{+K}(v)$ such that $u^{-} \notin X$, $u \in X$ and $(u^{-},u) \in A(D)$ and thus, $|N^+_{D}(u^-,X)|>0$.
This can be done by choosing $u$ be the closest vertex to $v$ in $B_{D}^{+(K+1)}(v) \cap X$ and $u^-$ be the in-neighbor of $u$ in $B_{D}^{+K}(v)$.
By the construction of $M^-(X)$, there exists an out-neighbor $u' \in X$ of $u^-$ (it is possible that $u'=u$) such that $u^-u' \in M^-(X)$ and thus $u^-u' \notin E(H')$.
As $u^- \in B_{D}^{+K}(v)$, we have $u' \in B_{D}^{+(K+1)}(v)\cap X$.
Therefore, $E(D[B_{D}^{+(K+1)}(v)])\nsubseteq E(H')$.
By the assumption of the lemma, we have $d(H')\leq d-\varepsilon'$.

(2) If $m\leq K/2$, then we obtain that
\begin{equation}\nonumber
\begin{split}
d(H')=&\frac{e_{H'}}{m-1}\leq\frac{d(m-1)-1/2}{m-1}=d-\frac{1}{2(m-1)}\leq d-\frac{1}{K}\leq d-\varepsilon',
\end{split}
\end{equation}
where the last inequality holds since $\varepsilon'\le 1/K$.
\end{proof}

The next lemma shows that if there exists such an $X$ in Lemma \ref{m1H}, then we can obtain the desired upper bound on the randomly perturbed threshold of $H$.

\begin{lemma}
\label{pt}
Let $d>1$ be a real number and let $1/n\ll\eta\ll \varepsilon'\ll 1/K\ll\gamma\ll\varepsilon,1/d,1/\Delta$.
Let $H$ be a graph with $\Delta(H)\leq\Delta$ and satisfying the following condition:
There exists an independent set $X\subseteq V(H)$ of size at most $\gamma n/2$ and an acyclic orientation $D$ of $H$ such that $m_1(H\setminus M^-(X))\leq d-\varepsilon'$.
If $G$ is an $n$-vertex graph with at least $\varepsilon n^2$ edges and $p=n^{-\frac{1}{d}-\eta}$, then w.h.p.  $G\cup G(n,p)$ contains a copy of $H$.
\end{lemma}
\begin{proof}
Fix the vertex set $X$ given in Lemma \ref{m1H} and let $E=M^-(X)$ to simplify the notation.
Note that $X$ is independent and thus, $E$ is the edge set of a family of  vertex-disjoint in-stars with roots in $X$ (See Figure~\ref{blue}).
Let $\overline{H}$ be the subgraph of $H$ obtained by removing the edges of $E$.
\begin{figure}[htbp]
  \centering
  \includegraphics[width=8cm]{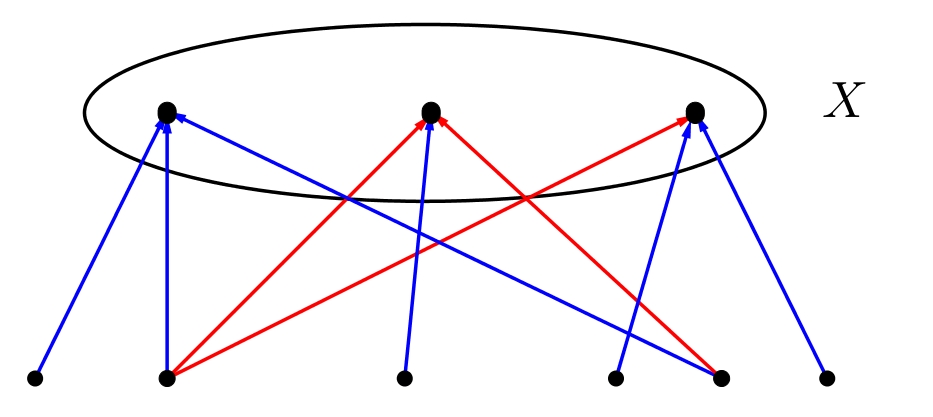}\\
  \caption{An example for $E=M^-(X)$, the blue edegs.}
  \label{blue}
\end{figure}


We define $G^c$ as a $2$-edge-colored $K_n$ on $V(G)$ where the blue edges are exactly the edges of $E(G)$.
Let $\Phi$ be the family of all embeddings of $H$ to $G^c$ such that all edges of $E$ are mapped to blue edges.
Then we note that if there is an embedding $\phi \in \Phi$ such that all the edges in $\phi(\overline{H})$ are present in $G(n,p)$, then $G\cup G(n,p)$ contains a copy of $H$.
We now show that the family $\Phi$ admits a $\frac{4}{\varepsilon n}$-vertex spread distribution.

Let $L=\{v\in V(G):d_G(v)\geq\varepsilon n/2\}$ be the set of high degree vertices.
As $2\varepsilon n^2\leq 2e_G\leq|L|\cdot n+\varepsilon n/2\cdot n$, we can obtain that $|L|\geq \varepsilon n$.
We now construct a random embedding $\varphi$ from $\{x_1,\dots, x_n\}$ to $V(K_n)$.
Label the vertices of $X$ as $X=\{x_1',\ldots,x_{|X|}'\}$ and let $Y_i=N^-_{D[E]}(x_i')$ for each $i\in[|X|]$.
For each $j\in [|X|]$, assume that $x_1',\ldots,x_{j-1}'$ and $Y_1,\ldots,Y_{j-1}$ are already embedded and let $I$ be the collection of the images.
We first choose $x_j'$ uniformly at random from $L\setminus I$ and there are at least $|L|-(\Delta+1) \gamma n/2 \geq \varepsilon n/4$ choices for the images of $x_j'$.
Label $Y_j=\{y_{j_1},\ldots,y_{j_{|Y_j|}}\}$. 
We choose the image of $y_{j_{\ell}}$ uniformly at random from $N_{G}(x_j')\setminus(I \cup \varphi(\{y_{j_1}, \ldots, y_{j_{\ell-1}}\}))$ for each $\ell\in[|Y_j|]$. By the same computation,  there are at least $\varepsilon n/2-(\Delta+1) \gamma n/2\geq\varepsilon n/4$ choices for the images of each $y_{j_{\ell}}$ where $\ell\in[|Y_j|]$.
After we embed all the vertices in $X \cup N_{D[E]}^-(X) = V(E)$, we embed the vertices in $V (H)\setminus V (E)$ to the remaining vertices of $V(G)$ uniformly at random.
This defines a probability distribution on $\Phi$, denoted by $\mu$.
We now prove that $\mu$ is a $\frac{4}{\varepsilon n}$-vertex spread distribution.

Let $C:=4/\varepsilon$.
Recall that in our random process of generating the embeddings of $H$ to $K_n$, we first choose the images of the vertices $V(E)$ and then assign images to the rest of the vertices.
For any $s\in[n]$ and any sequences $y_1,\ldots,y_s\in V(\overline{H})$, $z_1,\ldots,z_s\in V(K_n)$, let $t=|\{y_1,\ldots,y_s\}\cap V(E)|$.
If $t=s$, then $\mu(\{\psi:\psi(y_i)=z_i, \forall \ i\in[s]\})\leq\frac{1}{(\varepsilon n/4)^s}=\left(\frac{C}{n}\right)^s$; if $t<s$, then we have
\begin{equation}\nonumber
\begin{split}
\mu(\{\psi:\psi(y_i)=z_i, \forall \ i\in[s]\})&\leq\frac{1}{(\varepsilon n/4)^t\cdot(n-(\Delta+1)\gamma n/2)\cdots(n-(\Delta+1)\gamma n/2-(s-t-1))}\\
&=\left(\frac{C}{n}\right)^t\frac{(n-(\Delta+1)\gamma n/2-(s-t))!}{(n-(\Delta+1)\gamma n/2)!}\\
&\leq\left(\frac{C}{n}\right)^t\left(\frac{e}{n-(\Delta+1)\gamma n/2}\right)^{s-t} \le \left(\frac{C}{n}\right)^{s}.
\end{split}
\end{equation}
Therefore, $\Phi$ admits a $\frac{4}{\varepsilon n}$-vertex spread distribution.
By Proposition~\ref{kelly}, there exists $C'=C'(\varepsilon)$ and there is a $(C'/n^{1/m_1(\overline{H})})$-spread distribution on subgraphs of $K_n$ isomorphic to $H$ which admits a labeling such that all edges of $E$ are blue edges (edges of $G$).

By Lemma~\ref{m1H}, we obtain that $m_1(\overline{H})\leq d-\varepsilon'$ and therefore $p= n^{-\frac{1}{d}-\eta}\geq n^{-\frac{1}{d-\varepsilon'}+\eta}\geq n^{-\frac{1}{m_1(\overline{H})}}\log n$ by $1/n\ll\eta\ll\varepsilon' \ll1/d$.
By Proposition~\ref{FKNP}, we obtain that w.h.p. the random graph $G(n,p)$ contains a subgraph isomorphic to $\overline{H}$ such
that all edges in $E$ appear in $H$.
This yields a copy of $H$ in $G\cup G(n,p)$.
\end{proof}

The proof of Theorem~\ref{combin} immediately follows from the combination of Lemmas~\ref{m1H} and~\ref{pt}.

\section{Proofs of Theorems \ref{maindegenerate} and \ref{mainthm:regular}}\label{sec:profdegene}
\subsection{Proofs of Theorems \ref{maindegenerate} and \ref{mainthm:regular}}\label{sec:profdegene}

To prove Theorems~\ref{maindegenerate} and \ref{mainthm:regular}, it suffices to construct an acyclic orientation of $H$ satisfying the assumptions in Theorem~\ref{combin}. 
We first show that a $d$-degenerate graph with bounded maximum degree admits such an acyclic orientation.
\begin{proof}[Proof of Theorem \ref{maindegenerate}]
Let $1/n\ll\eta\ll\varepsilon'\ll1/K\ll\varepsilon,1/d,1/\Delta$.
Let $H$ be a $d$-degenerate graph with maximum degree at most $\Delta$.
Note that there is an ordering $(v_1,\ldots,v_n)$ of $V(H)$ such that for every $i$, $v_i$ has at most  $d$ neighbors among $\{v_1, \ldots, v_{i-1}\}$.
We define an orientation of $H$ as follows:
for each edge $v_iv_j,i<j$, orient the edge from $v_j$ to $v_i$.
The resulting digraph $D$ is acyclic and each vertex has out-degree at most $d$.
Let $V'=\{v\in V(H),|B_{D}^{+K}(v)|\geq K/2\}$ and $H'$ be a subgraph of $H$ of order $m$.
In addition, let $D'$ be the subdigraph of $D$ induced by $V(H')$.

(1) If $K/2<m\leq 2d\left((K+1)\Delta^{K+1}+1\right)$ then we note that the number of edges in $H'$ is at most $dm-\frac{d(d+1)}{2}$ since the first $d$ vertices has at most $d-1,d-2,\ldots,0$ out-neighbors in $H'$ while each other vertex has at most $d$ out-neighbors in $H'$.
Thus, we have 
\[
d(H')\leq\frac{dm-\frac{d(d+1)}{2}}{m-1}=d-\frac{\binom{d}{2}}{m-1}\leq d-\frac{\binom{d}{2}}{2d\left((K+1)\Delta^{K+1}+1\right)}\leq d-\varepsilon',
\]
where the last inequality holds since $\varepsilon'\ll1/K\ll1/\Delta,1/d$.

If $m\geq 2d\left((K+1)\Delta^{K+1}+1\right)+1$ and every vertex $v\in V(H')\cap V'$ satisfies $E(D[B_{D}^{+(K+1)}(v)])\setminus E(H')\neq\emptyset$, then we obtain the following observation:
\emph{For any vertex} $v\in V(H')$, \emph{there exists a vertex} $u\in B_D^{+(K+1)}(v) \cap V(H')$ \emph{such that} $d_{D'}^+(u)\leq d-1$.  
Indeed,
\begin{itemize}
    \item 
If $v \in V'$ and $B_{D}^{+(K+1)}(v) \subseteq V(H')$, then by the assumption, there exists an edge $(u,u^+)\in E(D[B_{D}^{+(K+1)}(v)])\setminus E(H')$. Then the vertex $u$ is as desired.
 \item  If $v \in V'$ and $B_{D}^{+(K+1)}(v) \nsubseteq V(H')$, then we choose a vertex $u \in V(H')$ such that there exists an out-neighbor $u^+$ of $u$ in $B_{D}^{+(K+1)}(v)$ with $u^+ \notin V(H')$. Thus, we have $d_{D'}^+(u)\leq d-1$.
 \item  If $v \notin V'$, then $|B_{D}^{+K}(v) \cap V(H')|\leq K/2 < d^K$.
Thus, there exists a vertex $u \in B_{D}^{+K}(v) \cap V(H')$ such that $d_{D'}^+(u)\leq d-1$. 
\end{itemize}

We now observe that each vertex $u$ is contained in at most $\sum_{i=0}^{K+1}\Delta^i\leq (K+1)\Delta^{K+1}+1$ sets of the form $B_{D}^{+(K+1)}(v)$ for some $v\in V(H')$. 
Thus, we have
\[
e_{H'}\leq\sum_{v\in V(H')}d_{D'}^+(v)\leq dm-\frac{m}{(K+1)\Delta^{K+1}+1}.
\]
Therefore, writing $K':=(K+1)\Delta^{K+1}+1$, the $1$-density of $H'$ can be bounded by 
\begin{equation}\nonumber
\begin{split}
d(H')&\leq\frac{(d-{1}/{K'})m}{m-1}\leq d-\frac1{K'}+\frac{d}{m-1}\leq d-\frac{1}{K'}+\frac{d}{2d K'}\leq d-\varepsilon',
\end{split}
\end{equation}
where the last inequality comes from $\varepsilon'\ll1/K\ll1/d,1/\Delta$.

(2) If $m\leq d$, then $e_{H'}\leq\binom{m}{2}\leq d(m-1)-1/2$.
If $d+1\leq m\leq K/2$, then we have
\[
e_{H'}\leq dm-\frac{d(d+1)}{2}\leq d(m-1)-\frac{1}{2},
\]
as $d\geq2$.

Therefore, $D$ satisfies the condition of Theorem~\ref{combin}, which completes the proof.
\end{proof}

We now prove Theorem~\ref{mainthm:regular}.
 

\begin{proof}[Proof of Theorem \ref{mainthm:regular}]
Choose the parameters such that 
$1/n\ll\eta\ll\varepsilon'\ll 1/K\ll\varepsilon, 
\gamma, 1/d$.
Let $H$ be as given in the statement of Theorem~\ref{mainthm:regular}.
It suffices to verify that $H$ satisfies the conditions of Theorem~\ref{combin} for $d/2$ under a suitable acyclic orientation.

We  first construct an acyclic orientation  $D$ of $H$. 
Note that each nonempty set of vertices of size at most $\gamma n$ has at least $d$ outgoing edges, so each connected component of $H$ has size at least $\gamma n$.
Consequently, the number of components $c$ satisfies $c\le 1/\gamma$. 
Let $C_1, \ldots, C_c$ be the components of $H$.
For each $1 \le i \le c$, choose an arbitrary vertex $r_i$ in $C_i$ and let $T_i$ be a spanning tree of $C_i$ rooted at $r_i$. 
Let $T$ be the union of all $T_i$. Then $T$ is a spanning forest of $H$ with $c$ trees.
We define the levels $L_0, \ldots, L_\ell $ of the spanning forest $T$ as follows. Let $L_0=\{r_1, \ldots, r_c\}$ be the set of roots. 
For $i\ge 1$, the level $L_i$ is the set of all 
vertices at distance $i$ from their respective roots within their own component tree of $T$. 
For the edges of $H$ between two different layers $L_i$ and $L_j$, orient each edge from the higher level to the lower level. i.e., if $i<j$, then orient the edge from $L_j$ to $L_i$.
For the edges within the same layer $L_i$, we choose an arbitrary acyclic orientation in each $L_i$.
This yields an acyclic orientation $D$ of $H$.

Let $V':=\{v\in V(H): |B_{D}^{+K}(v)|\geq K/2\}$ and let $H'$ be an $m$-vertex subgraph of $H$.  
We now verify that $H'$ satisfies conditions (1) and (2) of Theorem \ref{combin}.  

\noindent
\textbf{Verify Theorem \ref{combin} (1).}  
Suppose $m>K/2$ and for every vertex $v\in V(H')\cap V'$, we have $E(D[B^{+(K+1)}(v)])  \nsubseteq E(H')$.
We show that $d(H')\le d/2-\varepsilon'$. 

If $K/2<m<\frac{6(K+2)d^{K+2}}{\gamma} $, then $|\partial(V(H'))|\ge d+1$ as $\frac{6(K+2)d^{K+2}}{\gamma}\le \gamma n $. 
As $H$ is $d$-regular, it follows that  $e_{H'}\le \frac{dm-(d+1)}{2}=d(m-1)/2-1/2$. Thus, we obtain \[d(H')=\frac{e_{H'}}{m-1}\le d/2-\frac{\gamma }{12(K+2)d^{K+2}}\le d/2-\varepsilon',\] since $\varepsilon'\ll  1/K\ll \gamma, 1/d$.

We now  assume $m\ge \frac{6(K+2)d^{K+2}}{\gamma}$.
Note that by assumption for any  vertex $w\in V(H')\cap V'$, we have $E(D[B^{+(K+1)}(w)])  \nsubseteq E(H')$. 
Therefore,  there exists a vertex $u\in B_D^{+(K+1)}(w) \cap V(H')$ such that $u$ has degree at most $d-1$ in $H'$. 
Since $H$ is $d$-regular, for a fixed vertex $u$, the number of vertices $w$ such that $u\in B_D^{+(K+1)}(w)$ is at most 
  $\sum_{0\le i\le K+1} d^i\le (K+2)d^{K+1}=:K'$. 
  We now  count the number of such vertices $u$. 
Define    $V_0=\{u\in V(H'): |N(u)\cap V(H')|\le d-1\}$. Then, a double counting argument yields 
\begin{equation}\label{eq:lowdeg}
    |V_0|\ge  \frac{|V(H')\cap V'|}{K'}.
\end{equation}

Let $\overline{V'}=\{v\in V(H): |B_{D}^{+K}(v)|< K/2\}$ be the complement of $V'$. 
By the construction of the acyclic orientation $D$, we know that for any $v\in L_i$, $|B_{D}^{+K}(v)|\ge i$.
This implies 
$\overline{V'} \subseteq \bigcup_{0\le i< K/2} L_i$. 
Since $H$ is $d$-regular, we have $|L_i|\le cd^i\le d^i/\gamma $ and thus $|\overline{V'}|\le   \sum_{0\le i< K/2} |L_i|\le \frac{K+2}{2\gamma } d^{\frac{K}{2}}$. 
Together with $m\ge \frac{6(K+2)d^{K+2}}{\gamma}$, this yields that 
\begin{equation}\label{eq:badver}
    |V(H')\cap \overline{V'}|\le | \overline{V'}|\le \frac{K+2}{2\gamma }d^{\frac{K}{2}} \le \frac{m}{4d}.
\end{equation}
Note that $|V(H')\cap V'|=m-|V(H')\cap \overline{V'}|$. By combining (\ref{eq:lowdeg}) and (\ref{eq:badver}), we obtain
\begin{align}\label{eq:lowdegsize}
    |V_0|\ge \frac{m-|V(H')\cap \overline{V'}|}{K'}\ge \frac{m-m/(4d)}{K'}\ge \frac{m}{2K'}.
\end{align}
Observe that $e_{H'}\le \frac{dm-|V_0|}{2}\le d(m-1)/2-|V_0|/3$. 
By (\ref{eq:lowdegsize}),
we have 
\begin{align*}
     d(H')=\frac{e_{H'}}{m-1} \le d/2-\frac{|V_0|}{3(m-1)} \le d/2- \frac{1}{6K'}\le d/2-\eps'.
\end{align*}
The last inequality comes from $\varepsilon'\ll 1/K$. 

\noindent
\textbf{Verify Theorem \ref{combin} (2).} 
Assume $m\le K/2$. By the definition of $d(H')$, it suffices to consider the case $m\ge 2$. 
Note that $2\le m \le K/2 \le \gamma n$. Then 
 $|\partial(V(H'))|\ge d+1$. Since $H$ is $d$-regular, it follows  that  $e_{H'}\le \frac{dm-(d+1)}{2}=d(m-1)/2-1/2$.

Therefore, $D$ satisfies the condition of Theorem~\ref{combin} for $d/2$, which completes the proof.
\end{proof}

\subsection{Proofs of Corollaries \ref{cor:almost_2dreg} and \ref{cor:Hcyclepower}}\label{sec:cor}

To prove Corollary \ref{cor:almost_2dreg}, we need the following result.

\begin{theorem}{\rm (\cite[Theorem 7.32]{Bollobas2001random})}\label{thm:almostreg} For $d\ge 3$ and any fixed $a_0\ge 3$, almost every $d$-regular graph $G$ has the following property. If $V(H)=A\cup S \cup B$ with $a=|A|\le |B|$ and no edges between $A$ and $B$, then \[ |S| \ge \begin{cases} d, & \text{if } a = 1;\\ 2d-3, & \text{if } a = 2;\\ (d-2)a, & \text{if } 3 \le a \le a_0;\\ (d-2)a_0, & \text{if } a \ge a_0. \end{cases} \] \end{theorem}

We now prove Corollary \ref{cor:almost_2dreg}. 


\begin{proof}[Proof of Corollary \ref{cor:almost_2dreg}]
We will show that almost every $d$-regular graph $H$ satisfies 
$|\partial(X)| \ge d+1$ for all $X \subseteq V(H)$ with $2 \le |X| \le n/2$. 
Taking $\gamma = 1/2$ in Theorem \ref{mainthm:regular} then yields the desired conclusion. 

Let $H$ be an $n$-vertex $d$-regular graph. Fix any subset $X\subseteq V(H)$ with $2\le |X|\le n/2$. Define
        $Y = N(X) = \{v\in V(H)\setminus X : \exists u\in X \text{ such that } uv\in E(H)\},$
    and set $Z = V(H)\setminus (X\cup Y)$. Then $V(H)$ is partitioned into three disjoint parts $X\cup Y\cup Z$, and there are no edges between $X$ and $Z$.  Set $a = \min\{|X|,|Z|\}$. If $a=1$, then $|X|> |Z|=1$. Since $n$ is sufficiently large, $|Y|=n-|X|-|Z|\ge n/2-1\ge d+1$.
    Now assume $a\ge 2$ and recall that $d\ge 4$. By Theorem \ref{thm:almostreg} with $(d, \{X, Z\}, Y)$ playing the role of $(r, \{A, B\}, S)$,
we obtain that for almost every $d$-regular graph $H$,
     \[
        |Y| \ge 
        \begin{cases}
            d+1,    & \text{if } a = 2,\\
            d+2,    & \text{if } a \ge 3.
        \end{cases}
    \]

 Since every vertex in $Y$ is adjacent to at least one vertex in $X$, we have $|\partial(X)|\ge |Y|$. 
    Hence $|\partial(X)|\ge d+1$ whenever $2\le |X|\le n/2$. This completes the proof. 
 \end{proof}

\noindent
\textbf{Remark.} For $|X|=2$, degree counting gives $|\partial(X)|\ge 2d-2\ge d+2$, using $d\ge 4$. Consequently, for $d\ge 4$, the  nontrivial edge cut in almost every $d$-regular graph has size at least $d+2$.
However, the bound $|\partial(X)|\ge d+1$ already suffices for our purpose.

We now prove Corollary \ref{cor:Hcyclepower}. The corollary follows from Theorem \ref{mainthm:regular} once we show that the $d$-th power of a Hamiltonian cycle satisfies the $2d$-regular graph family condition stated in Theorem \ref{mainthm:regular}.

\begin{proof}[Proof of Corollary \ref{cor:Hcyclepower}]
Choose $1/n<\gamma \ll 1/d$. Let $H$ be the $d$-th power of a Hamilton cycle. 
By Theorem \ref{mainthm:regular}, it suffices to show that there exists a constant $\gamma >0 $ such that   $H$ satisfies 
 $|\partial(X)|\ge 2d+1$ for every $X\subseteq V(H)$ with $2\le |X|\le \gamma  n $.  


    The conditions $|X|\le \gamma n $ and $\gamma \ll 1/d$ imply that  $H[X]$ is contained in the $d$-th power of a path $P$. 
    Let us label $V(P) \cap X=\{v_1,\ldots, v_{|X|}\}$ in the natural order. If $|X|=2$, then $|\partial(X)|\ge 4d-2>2d+1$ for $d\ge 2$, as desired. Now suppose $|X|\ge 3$. 
    In the $d$-th power of a path $P$, the ends $v_1$ and $v_{|X|}$ have  degree at most $d$ in $H[X]$, and the vertex $v_2$ has degree at most $d+1$ in $H[X]$. 
    Thus, $|\partial(X)|\ge \sum_{i \in \{1,2, |X|\}} (2d-|N(v_i)\cap X|)\ge 3d-1 \ge 2d+1 $ for $d\ge 2$.
    Therefore,  the $d$-th power of a Hamilton cycle  $H$ satisfies the condition of Theorem \ref{mainthm:regular}, which completes the proof.
\end{proof}


\section{Concluding Remarks}
We study perturbation thresholds for spanning subgraph containment in the randomly perturbed model, where an arbitrary deterministic graph with at least $\varepsilon n^2$  edges is augmented by the addition of a binomial random graph. Throughout, all considered graphs are assumed to have bounded maximum degree.
Our main result, Theorem \ref{maindegenerate}, shows that for every fixed $\varepsilon>0$, there exists $\eta >0$ such that for all  $d$-degenerate graphs, the perturbation threshold in the randomly perturbed model is at most
$n^{-1/d-\eta}$.
This gives a polynomial improvement over the corresponding threshold $p\le n^{-1/d}$ in the binomial random graph model (shown by Riordan~\cite{MR1762785} for $d\ge 3$ and by Chen, Han, and Luo~\cite{chen2024thresholds} for $d=2$).

For some natural subclasses, the exponents in the two models already align. 
For example, every $K_4$-minor-free graph (in particular, every outerplanar graph)  is $2$-degenerate, by a classical result of Dirac \cite{Dirac1960}. Hence, our bound yields a perturbation threshold of at most $n^{-1/2-\eta}$, whereas in the binomial random graph the threshold is at most $n^{-1/2}$.

However, the situation is different for planar graphs. Planar graphs are $5$-degenerate, so applying Theorem \ref{maindegenerate} with $d=5$ would only give a perturbed threshold at most $n^{-1/5-\eta}$. Yet the  known threshold in $G(n,p)$ for containing a planar graph of bounded maximum degree is at most $n^{-1/3}$
(due to Riordan  ~\cite{MR1762785} and Chen, Han and Luo~\cite{chen2024thresholds}). Thus, the existence of a polynomial saving in the perturbed setting for planar graphs remains open. 

\begin{problem}
    For  constants $\varepsilon, \Delta > 0$, does there exist $\eta>0$ such that for all sufficiently large $n$, if $G$ is an $n$-vertex graph with at least $\varepsilon n^2$ edges and $H$ is an $n$-vertex planar graph with maximum degree at most $\Delta$, then $G \cup G(n, n^{-1/3-\eta})$ w.h.p. contains a copy of $H$? 
\end{problem}

Another natural but challenging problem is to investigate how large the $\eta$ in this type of results can be, e.g., in Theorem~\ref{maindegenerate}.
This has been studied for square of Hamilton cycles under the minimum degree condition by B\"ottcher, Parczyk, Sgueglia and Skokan \cite{bottcher2024square}, that is, given graph $G$ with minimum degree $\alpha n$, they determined the best possible $\eta=\eta(\alpha)$ for $p=n^{-1/2-\eta}$ in the perturbation threshold.

\section*{Acknowledgements}
JH was partially supported by the Natural Science Foundation of China (12371341).
SI was supported by the National Research Foundation of Korea (NRF) grant funded by the Korea government(MSIT) No. RS-2023-00210430, and supported by the Institute for Basic Science (IBS-R029-C4).
 JZ was supported by the China Postdoctoral Science Foundation (No. 2024M764113).

\bibliographystyle{plain}
\bibliography{ref1}
\end{document}